\documentclass{article}

\usepackage[utf8]{inputenc}
\usepackage{amsmath}
\usepackage{ifsym}
\usepackage{amsfonts}

\usepackage{amsthm}
\newtheorem{thm}{Theorem}
\newtheorem{defi}{Definition}

\newtheorem{lemat}{Lemma}

\title{On the isomorphisms of Fourier algebras of finite abelian groups}
\author{A. Czuron
\&
M. Wojciechowski\thanks{The research of this author has been supported by NCN grant no. N N201 607840}\\
\\
Institute of Mathematics, Polish Academy of Sciences\\
00-956 Warszawa, Poland\\
\&\\
The University of Warsaw\\
ul. Banacha 2, 02-097 Warszawa, Poland\\
\\
E-mail: alanczuron@gmail.com,\\ \qquad miwoj-impan@o2.pl}
\begin{document}

\maketitle

\section{Introduction}

%Let $m=(m_0,m_1,...)$ be a sequence of natural number`s, where $m_k \ge 2$. By $Z_{m_k}$ the cyclic group with $m_k$ elements.
%\\

%Let $G_m = \bigoplus Z_m$ be direct sum.  Notice that the dual groups to the Cantor groups 
%$(Z_{m})^{\omega}$ are groups $(Z_{m})^{\mathbb{N}}$. It is known, that characters on 
%$G_m$ forms complete orthonormal system $\widehat{G_m}$ in $L^1$ sequence space. By $\widehat{G_m}$ we denote a discrete  group dual to $G_m$. Elements of $\widehat{G_m}$ are called characters. They are indexed by finite sequences
%$(n_j)_{j=1}^N$, $N\in \mathbb{N}$, where $0\leq n_j\leq m_j$ an given by

%$$ \psi_n = \Pi_{k=0}^{N} r_{k}^{n_k}$$
%where $r_n (x)=exp(\frac{2\pi ix_n}{m_n})$.
%Such systems of functions are called Vielenkin systems.
%\\

%If $G'$ is a coset of a group $G$ by $coset(G')$ we denote the set of finite cosets  of $G$
%which are contained in $G'$. By $coset_r(G')$ we denote cosets from $G'$ of cardinality $r$.
%\\

In this paper we consider spaces  $L^1 (G)$ for compact abelian group G where the integral norm, denoted by $\| \|$, is taken with respect to the Haar measure.
More specific  we are studying algebraic homomorphisms between the Fourier algebras
of the form $A(\widehat{G})$. Our main result is the following

\begin{thm}
Let $G$ be commutative compact torsion group with finite exponent. Suppose
that $G$ does not contain a subgroup isomorphic to $Z_{p^a}^\omega$. Then there exists
$C>0$ such that for every isomorphism $T: A(Z_{p^\alpha}^n)\to H\subset A(\widehat{G})$ where
$H$ is a subalgebra of  $A(\widehat{G})$ there is
$$
\|T\|\cdot\|T^{-1}\|> \Phi_G(n).
$$
where $\lim_{n\rightarrow \infty}\Phi_G(n)=\infty$. Additionaly, if $G$ is a $p$-group then
$$
\Phi_G(n) >C(\log \log n)^{1/4}
$$
and if $G$ has no  $p$-subgroup then
$$
\Phi_G(n)>C(\log\log\log n)^{1/4}
$$
\end{thm}

Any homomorphism between the Fourier algebras on abelian locally compact groups is given by the corresponding mapping betwen the groups (cf. [Ru] for this fact and, more general, for notations and definitions used in this paper). This motivates the following definition and the next reformulation of our result.

\begin{defi} Let $X$, $Y$ be Banach spaces. Subsystem $\{f_1,f_2,...,f_n \} \subset (f_i)$, where $(f_i)$ is system of elements of $X$ is $C$-representable in system (${g_m}$) in $Y$ if there exist an injection $\sigma_n$ s.t:
$$\sigma_n: \{ f_1,f_2,...,f_n \} \rightarrow \{ g_i : i \in \mathbb{N} \}$$
and
$$C^{-1} \|\sum_{j=1}^{n} a_k f_j \| \le \| \sum_{j=1}^{n} a_k \sigma_n (f_k) \| \le 
C \| \sum_{j=1}^{n}  a_k f_j\|$$
for every choice of scalars $a_k$. We will use symbol $\prec^{C}$ for $C$-representation property. If a sequence of incjections exists which gives the representation of increasing sequence of finite subsystems with uniformly bounded constnt $C$, we say that $(f_n)$ is $C$-locally represented in $(g_n)$. For $C$-locally representarion we use the symbol
$\prec^C_{loc}$
\end{defi}
In particular if $\sigma_n$'s are the restrictions of a fixed  bijection 
$\sigma:\mathbb{N}\rightarrow\mathbb{N}$, we get a definition of permutation equivalence of systems. In the spirit of the above definition Theorem 1 translates onto the property of
the representation property of Vilenkin system  on groups:
$\bigoplus_{k=1}^{N} Z_{p_{k}^{s_k}}^\mathbb{N}$. (actually any abelian torsion group with finite exponent has this form - c.f. [Ro], Corollary 20.37)

\begin{thm}
Suppose, that $\widehat{G}$, $\widehat{H}$ are dual groups to the Cantor groups 
$\oplus_{i=1}^{N}( Z_{p^{s_i}_i})^{\omega}$ and $\oplus_{k=1}^{M}( Z_{q^{r_k}_k})^{\omega}$.
 Then the Vilenkin system on $\widehat{G}$ is locally representable in Vilenkin system on group $\widehat{H}$ iff for every $p_{i}^{s_i} \in \{p_{1}^{s_1},p_{2}^{s_2},...,p_{k}^{s_k}\}$ there exist $q_{j}^{r_j} \in \{q_{1}^{r_1},q_{2}^{r_2},...,q_{k}^{r_l}\}$
s.t: $p_{i}=q_{j}$ i $s_i \le r_j$
\end{thm}

The problem of comparision of the Walsh system with trigonometric system in $L^p$ norms was asked by A. Pelczynski [P] and was solved in non-reflexive case in [W]. This paper 
deal with more delicate case of characters on general bounded Vilkenkin groups. Our solution uses two very powerful results from harmonic analysis and number theory.
The first one is a quantitative version of Helson theorem (cf. [GS]):

\begin{thm}(Green-Sanders): Let $\mu$ be a measure. Then $\mu$ is idempotent ($\widehat{\mu}$ is characteristic function on $\widehat{G}$) if and only if:
$\left\{ \gamma \in \widehat{G} : \mu (\gamma)=1 \right\}$ lies in the coset ring $\widehat{G}$, i.e.

$$\widehat{\mu}=\sum_{j=1}^{L} \epsilon_j1_{\gamma_j +\Gamma_j}
$$
where $\Gamma_j \in \widehat{G}$ are open subgroups, $\gamma_j \in \widehat{G}$ are elements of group and $\epsilon_j$'s are signs. Constant L in this sum is bounded by $e^{e^{C{\|{\mu}\|}^4}} $ and the number of  $\Gamma_j$ is less then $\|{\mu}\| + \frac{1}{100}$.

\end{thm}

The second one is an explicit bound on the number of solutions of S-unit equations
(cf. [E1]):

\begin{thm}(Evertse) Let $M$ be a finite set of rational primes. Then the equation:

$$
x_1 + x_2 + ... + x_{n+1}=0
$$
where every $x_i$ is formed from product of prime numbers from the set $M$,
\\
$gcd(x_1,x_2,...,x_{n+1})=1$, and there is no vacous sums
$x_{i_1}+x_{i_2}+...+x_{i_{m}}$, has the number of integral solutions $\textbf{x}=(x_1, x_2,..., x_{n+1})$ bounded by $C_1\cdot\exp( C_2 n^3\log n))$ where the constants $C_1$ and $C_2$
depend on $M$ only.
\end{thm}

{\bf Remark.} We do not know whether the asumption of bounded exponent is significant.
Evidently the method used in our proof does not work in that case. However we do not know any counterexample neither.

{\bf Acknowledgments.} The authors thank Prof. A Schinzel for valuable help during the work on this paper.

\section{Auxiliary Results}

If $G'$ is a coset of a group $G$ by $coset(G')$ we denote the set of finite cosets  of $G$
which are contained in $G'$. By $coset_r(G')$ we denote cosets from $G'$ of cardinality $r$.

Put $G=Z_{p^{\alpha}}^{N}$ and $H=(Z_{p^{\beta}})^{\mathbb{N}}$  and suppose that $G\prec^{c} H$.
\begin{lemat}
For every $G' \in coset(G) $ there exists $H' \in coset(H)$ s.t. 

$$
H'\subset \sigma_{n}(G')
$$
and $\# H' \ge \# G' p^{-\Lambda}$
where $\sigma_n$ are injections from the $c$-representation, and $\Lambda$ is constant depending on $(p, L, c)$, where $L$ comes from Green-Sanders theorem and  $c$ - from locally representation. Similarly if $H' \subset coset(H)$ and $H' \subset \sigma_{n} (G')$ then there exists $G'' \subset coset(G)$ s.t.

$$
G'' \subset \sigma^{-1}_{n} (H')
$$
and $\# G'' \ge \# H' p^{-\Lambda}$
\end{lemat}
\begin{proof}

%Suppose that $\sigma_n$  are elements of familly of injections which realise the c-locally representation. 
Let $G' \in coset(G)$. Then 
$$\widehat{\pmb{1}_{G'}}=\sum_{\chi_k \in G'} \chi_k$$
and

$$\| \widehat{\pmb{1}_{G'}} \| =1$$
Hence
$$T_{\sigma_{n}}(\widehat{\pmb{1}_{G'}})=\sum_{\chi_k \in G'} \chi_{\sigma_n (k)}$$ 
is an idempotent in the convolution algebra $L^1 (H)$. Moreover 
$$\| T_{\sigma_n} (\widehat{\pmb{1}_{G'}}) \| \le C\| \widehat{\pmb{1}_{G'}}\| = C$$
In this case by the theorem of Green-Sanders there are cosets $A_{1}, A_{2},...,A_{l_1}$ and $B_{1},B_{2},\dots, B_{l_2}$ s.t, 
$$
\sigma_n(G')=U
$$
where 
$${\pmb{1}_{U}}= \sum_{i=1}^{l_2} \pmb{1}_{A_i} - \sum_{j=1}^{l_1}\pmb{1}_{ B_j}$$
and $l_1,l_2 \le L$ where $L=e^{e^{D C^4}}$, where $D$ is universal constant from Green-Sanders theorem, and C comes from c-locally representation property. Similarly for every coset $H'$ s.t, $H' \subset \sigma_n (G^`)$ we have:
$$
\sigma^{-1}_{n} (H') = V
$$
where
$$
{\pmb{1}_{V}}= \sum_{i=1}^{l'_2} \pmb{1}_{A'_i} - \sum_{j=1}^{l'_1}\pmb{1}_{ B'_j}
$$
for some cosets $A'_1,A'_2,\dots,A'_{l'_1},B'_{1},B'_{2},\dots,B'_{l'_2}\subset G'$ and $l'_1,l'_2 \le L$ where $L=e^{e^{D C^4}}$.

\end{proof}

The proof lemma 1 will be finished by applying the next:
\begin{lemat}
Let $\pmb{1}_{U} = \sum_{i=1}^{l_1} \pmb{1}_{A_i} - \sum_{j=1}^{l_2} \pmb{1}_{B_{j}}$ where $A_i, B_j \in coset(G)$. Suppose that $\# U= p^K $ and $l_1, l_2 \le L$. Then there exists $\Lambda=\Lambda (L,p)$ and 
$G' \in coset_s(G) $

s.t.:
$$
G' \subset U
$$

and
$$
\# G \ge p^{K-\Lambda}
$$

We can take $\Lambda(L,p)=L+\log_p L$.
\end{lemat}
\begin{proof} We can assume that $B_j$'s are ordered according to non increasing cardinalities.
It is clear that  $\sup_i  \# (A_i - \cup_{i} B_i) \ge {p^K}{L^{-1}}. $ Suppose that the supremum is attained in $A_1$. So we have 

$$
p^{k_0}=\# A_1 \ge {p^K}{L^{-1}}=p^{K-\log_p{L}}
$$

We define by induction the decreasing sequence of cosets
$A_1=Z_1\supset Z_2\supset\dots\supset Z_r$ such that

$$
Z_k\cap (B_1\cup B_2\cup\dots\cup B_k)=\emptyset
$$

and

$$
Z_k\not\subset \bigcup_{j=1}^r B_j
$$

and

$$
\# Z_k\geq p^{k_0-k}
$$

Suppose that $Z_1,Z_2,\dots, Z_k$ are already defined. Then $Z_k\setminus B_{k+1}\neq\emptyset$. Therefore 

$$
\#Z_k\cap B_{k+1}\leq p^{k_0-k-1}.
$$

Hence there exist a coset $S\supset B_{k+1}$ with $\#S=p^{k_0-k-1}$ and the decomposition

$$
Z_k\setminus B_{k+1}=Z'_1\cup Z'_2\cup\dots\cup Z'_s \cup (S\setminus B_{k+1})
$$

onto pairwise disjoint cosets $Z'_j$, $j=1,2,\dots,j$ with $\# Z'_j= p^{k_0-k-1}$ for $j=1,2,\dots,s$.
and, moreover, $S$ could be chosen in such a way that at least  one of $Z'_j$'s does not contain in $\bigcup_{j=1}^s B_j$.
Then we choose it to be $Z_{k+1}$.
\end{proof}

\begin{lemat}
Let $G< (Z_{p^{\alpha}})^{N}$. If for given natural number $k$, group $G$ does not contain a subgroup isomorphic to $(Z_{p^{\alpha}})^{k+1}$ then 

$$
 \# G< p^{\alpha k}p^{(\alpha - 1)(N-k)}
$$
\end{lemat}

For the sake of completenes we present the proof of Lemma 3. It follows immediately from
the general form of a subgroup of $(Z_{p^\alpha})^N$.

\begin{lemat}
Every subgroup of $\oplus_{i=1}^N(Z_{p^{\alpha_i}})$ is isomorphic to
$\oplus_{i=1}^M (Z_{p^{\beta_i}})$ for some $M\leq N$ and $\beta_i\leq\alpha_i$ for
$i=1,2,\dots,M$.
\end{lemat} 

\begin{proof} We adopt definitions and notations from [ARS]. Consider 
$\oplus_{i=1}^N(Z_{p^{\alpha_i}})$ as a left module over $Z_p$ with natural
multiplication. Recall that a module $M$ has the finite length provided there exists a composition sequence 
$$0\subset M_1 \subset M_2 \subset M_3 \subset ... \subset M_n \subset M$$
of submoduls such that each quotient $M_{k+1}/M_k$ is simple, i.e. it does not contains non-zero proper submoduls. Obviously $\oplus_{i=1}^N Z_{p^{\alpha_i}}$ has finite length. Basic properties of composition sequences are:

(i) for any simple module $L$ the number of quotients $M_{k+1}/M_k$ isomorphic to $L$
does not depend on choise the composition sequence (it is usually not unique); we denote this number $\ell(L,M)$.

(ii) for every exact sequence of moduls of finite length
$$0 \rightarrow M' \rightarrow M \rightarrow M'' \rightarrow 0$$ and every simple module $L$, there is $\ell(L,M')+\ell(L,M'')=\ell(L,M)$

Clearly $\ell(Z_p,\oplus_{i=1}^N Z_{p^{\alpha_i}}) =\sum_{i=1}^N\alpha_i$ and
$\ell(L,\oplus_{i=1}^N Z_{p^{\alpha_i}})=0$ for $L\neq Z_p$.
Suppose now that $\oplus_{i=1}^N Z_{p^{\alpha_i}}$ contains a subgroup
isomorphic to $\oplus_{i=1}^M Z_{p^{\beta_i}}$ for some $M>N$. Then obviously it
contains a subgroup isomorphic to $\oplus_{i=1}^M Z_{p^{}}$.
Let 
$$0 \rightarrow K \rightarrow \oplus_{i=1}^{N} (Z_{p^{\alpha_i}}) \rightarrow L 
\rightarrow 0$$
be an exact sequence where third arrow denotes the homomorhpism 
$$\Lambda_p : \oplus_{i=1}^N Z_{p^{\alpha_i}}\to \oplus_{i=1}^N Z_{p^{\alpha_i-1}}$$ 
given by coordinatewise multiplication by prime $p$, and $K$ is a kernel of $\Lambda_p$.
Obviosly $\oplus_{i=1}^M Z_{p^{}}\subset K$. Therefore, by (ii),
$$
\aligned
\ell(Z_p,K) = &\ell(Z_p,\oplus_{i=1}^N Z_{p^{\alpha_i}})
-\ell(Z_p,\oplus_{i=1}^N Z_{p^{\alpha_i-1}}) \\=&\sum_{i=1}^N\alpha_i-\sum_{i=1}^N(\alpha_i-1)=N
\endaligned
$$
This contradicts our assumption since $M=\ell(Z_p,\oplus_{i=1}^M Z_{})<\ell(Z_p,K)$.
\end{proof}

\begin{lemat}
For $K=\frac{N}{R}$, where $R \in \mathbb{N}$ and $M\leq K$ we have:

$$
\# coset_{p^M} (\mathbb{Z}_{p^{\alpha}}^{N})< p^{\frac{N^2}{R}+\frac{N^2}{R^2}+O(N)} 
$$

If $\pmb{\gamma}=(\gamma_1,\gamma_2,...)$ satisfy $\gamma_i \le \alpha - 1$ for $j=1,2,\dots$ and $\sum\gamma_i =N\alpha$, then

$$
\# coset_{p^{K}} (\oplus Z_{p^{\gamma_j}})> p^{\frac{\alpha}{\alpha-1}\frac{N^2}{R} - 2\frac{N^2}{R^2} + O(N)}
$$

\end{lemat}
The proofs of Lemma 5 is given in the Appendix.

\section{Proof of Theorem 1}

\begin{proof}

It is enough to show, that $(Z_{p^l})^{\mathbb{N}} \not\prec_{loc} \oplus_{j=1}^{M} (Z_{(q_j)^{s_j}})^{\mathbb{N}}$ if either:
\\

1) $gcd(p, q_j)=1$ for $i\in {1,2,...,M}$
\\

2) $p=q_j$ but $s_j<l$ whenever $p=q_j$.
\\

We begin with Case 1.
Suppose to the  contrary that $(Z_{p^l})^{\mathbb{N}} \prec^{C}_{loc} \oplus_{j=1}^{M} (Z_{(q_j)^{s_j}})^{\mathbb{N}}$. Let $\sigma_{n}:(Z_{p^l})^{\mathbb{N}} \rightarrow \oplus_{j=1}^{M} (Z_{(q_j)^{s_j}})^{\mathbb{N}}$ be injections which realise the $c$-local representation.
Let $\Gamma \subset (Z_{p^l})^{\mathbb{N}}$ be any finite coset. Then $\widehat{\pmb{1}}_{\Gamma} = \sum_{\chi_k \in \Gamma} \chi_k$ and,
by Lemma 1 we have:

$$
\sigma_n(\Gamma)=U
$$
where

\begin{equation}\label{1}
{\pmb{1}_{U}}= \sum_{i=1}^{l_2} \pmb{1}_{A_i} - \sum_{j=1}^{l_1}\pmb{1}_{ B_j}
\end{equation}
and $l_1,l_2 \le L$ where $L=e^{e^{D C^4}}$ and $A_1,A_2,...,A_{l_1},B_1,B_2,...,B_{l_2}$ are finite cosets of $\oplus_{j=1}^{M} (Z_{(q_j)^{s_j}})^{\mathbb{N}}$.
Since the cardinality of any finite coset of the group  $ \oplus_{j=1}^{M} (Z_{q_j ^{s_j}})^{\mathbb{N}}$ is an integer of the form $q_{1}^{\alpha_1}q_{2}^{\alpha_2}...q_{M}^{\alpha_M}$ and, similarly, cardinality of $\Gamma$ is an integer of the form $p^R$, then any representation $\eqref{1}$
 gives an equation

$$
\# U=\# A_1 + \# A_2 +...+ \# A_{l_1} - \# B_1 -\# B_2 -...- \# B_{l_2}
$$
which is equivalent to:
\begin{equation}\label{2}
 p^R = x_1+x_2+...+x_l
\end{equation}
 where $l\le 2L$ and every $|x_i|$ is of the form $q_{1}^{\alpha_1}q_{2}^{\alpha_2}...q_{M}^{\alpha_M}$.
Clearly for every fixed $R$ there exists a tuple $(x_1,x_2,\dots,x_n)$ satisfying  $\eqref{2}$
without vacous subsums.
By  Theorem 4 for fixed $L$ such equation has no more then 
$$
C_1\exp( C_2 L^3\log{L})<C_3\exp\exp\exp DC^4
$$ 
solutions. Therefore,
since $R$ may be arbitrary integer less then $n$ and for every such number we get different solution of $\eqref{2}$, we get
$$
C>C_4\cdot (\log\log\log n)^{{1/4}}
$$ 
\\

The Case 2 is devided on  2 parts. 
 \vskip3mm
 Part A. We show that $(Z_{p^{\alpha}})^{\mathbb{N}} \not\prec_{loc} (Z_{p^{\beta}})^{\mathbb{N}}$ for $\alpha > \beta$. Suppose to the contrary that $(Z_{p^{\alpha}})^{\mathbb{N}} \prec^{C}_{loc} (Z_{p^{\beta}})^{\mathbb{N}}$ for some $C>0$. Suppose that $\sigma_n$ are injections from the defintion of local representation.

We begin our consideration with the group $G'= (Z_{p^{\alpha}})^{N} < (Z_{p^{\alpha}})^{\mathbb{N}} $. Applying Lemma 1 for $G'$ we know that there exists $H' \in  coset((Z_{p^{\beta}})^{\mathbb{N}} )$ s.t. $H'\subset \sigma_{n}(G')$ and $\# H' \ge p^{N\alpha -\Lambda}$, for some constant $\Lambda= \Lambda(p,L)$. 
We put
\begin{equation}\label{y}
Y= \bigcup_{j=p^{K-\Lambda}}^{p^K} coset_j (G').
\end{equation}
for give natural $ K$ and
$$
X=\{ A\subset Y:  1\le \# A \le L, \sigma_{N}(A)\subset coset_{p^K}((Z_{p^{\beta}})^{\mathbb{N}} \cap H' \}.
$$
 For $a \in X$ we define 
$$
P(a)= \bigcup_{y\in a} y. 
$$
Clearly
$$
X=\bigcup_{\Gamma \in Y} \{ a\in X: \Gamma \subset P(a) \}
$$
which yields
\begin{equation}\label{3}
\# X \le \# Y \cdot\sup_{\Gamma \in Y} \# \{ a\in X: \Gamma \subset P(a) \}.
\end{equation}
By $\eqref{y}$ we know that $\# Y$ is bounded by the number of non-empty summands in $\eqref{y}$, that is $\Lambda$, multiplied by the quantity provided by Lemma 5. So we get:
\begin{equation}\label{32}
\# Y  \le \Lambda p^{\frac{N^2}{R}+\frac{N^2}{R^2}+O(N)} 
\end{equation}
Now we estimate other cardinalities appearing in $\eqref{3}$.
Let $H'' \subset coset( H')$ with $\#H'' =p^{K}$ then, applying Lemma 1 twice, we select $G'' \subset coset (G')$ with $\# G'' \ge p^{K-\Lambda}$ s.t.  $G'' \subset \sigma^{-1}_n (H'')$  and $H''' \in coset(H'')$ with $\# H''' \ge p^{K-2\Lambda}$ s.t. $H''' \subset \sigma_{n} (G'')$. It follows that if the preimage  $\sigma^{-1}_{n} (\widetilde{H})$ of any $\widetilde{H} \in coset(H')$ contains  $G''$ then $H'''\subset \widetilde{H}$. Then to estimate the number of possible $a\in X$, s.t. $ G' \subset P(a)$ it sufficies to estimate the number of posible costes $\widetilde{H} \in coset (H')$ s.t. $H''' \subset 
\widetilde{H}$ and $\#\widetilde{H} \le p^K$. The above property does not depend on the choice of 
$G'$. So to estimate our supremum we have to estimate the number of cosets 
$\widetilde{H} \in coset(H')$ s.t. $H''' \subset \widetilde{H}$ and $\# \widetilde{H} \le p^{K}$  where

$$
H''' \in \bigcup_{j=p^{K-2\Lambda}}^{p^K} coset_j (H')
$$ is an arbitrary coset.

It sufficies to consider the case where $H'''$ is a subgroup. Then the coset $\widetilde{H}$ containing $H'''$ is also a subgroup and it is generated by $H'''$ and at most $2\Lambda$ elements of $H'$ . Hence the number we are looking for does not exceed the number of subsets of $H'$ with cardinality smaller then $2\Lambda$ which equals:

$$
\sum_{j=1}^{2\Lambda} {\#H' \choose j} \le \sum_{j=1}^{2\Lambda} {p^{N \alpha } \choose j} \le
2\Lambda p^{N \alpha \Lambda}
$$
Hence, by \eqref{3} and \eqref{32},
\begin{equation}\label{44}
\# X<
2\Lambda p^{2N \alpha \Lambda} p^{\frac{N^2}{K} + 2\frac{N^2}{K^2} + o(N)}=
p^{\Psi(L)}\cdot p^{\frac{N^2}{K} + 2\frac{N^2}{K^2} + O(N)}.
\end{equation}
where, by Lemma 2,
\begin{equation}\label{66} 
\Psi(L)=\Lambda(L,p)\cdot 2N\alpha \log_p 2\Lambda(L,p) < C\cdot L\cdot N.
\end{equation}

On the other hand we know that to different $\widetilde{H} \in coset(H')$ s.t. $\#\widetilde{H} = p^K$ corespond different elements of $a\in X$ s.t. $P(a)=\sigma^{-1}(\widetilde{H})$ . Therefore 

\begin{equation}\label{33}
\#X>\# coset_{p^K} (H').
\end{equation}
By Lemma 4, for sufficiently big  $K=\frac{N}{R} $, where $R\in \mathbb{N}$, the cardinality of $coset_{p^K} (H')$ satisfies:
\begin{equation}\label{ca}
\# coset_{p^{K}}(H') > p^{(\frac{\alpha}{\alpha-1}) \frac{N^2}{K} - 2\frac{N^2}{K^2}+O(N)}
\end{equation}
By \eqref{44}, $\eqref{33}$ and $\eqref{ca}$ we get
$$
\Psi(L)>\frac1{\alpha-1}\frac{N^2}{K}-4\frac{N^2}{K^2}+O(n)
$$
which, by \eqref{66} leads to
$$
C\cdot L>\big(\frac1{\alpha-1}\frac1{K}-4\frac{1}{K^2}\big)\cdot N
$$
For suitable choice of $K$ we get
$$
L>C_1\cdot N
$$
where  the constant $C_1$ depends on $p$ and $\alpha$ only. 
Since by Theorem 3, $L< \exp(\exp (c_2 C^4))$ the required estimates follows.

%Substituting $r=\beta^{*}_{1}=\frac{N}{K}$ and dividing right hand side of equation .. by .. we get that number of subgroups of rank $p^r$ contained in $A_2$ divided by the number of subgourps of the same of less rank contained in $A_1$ is bigger then:
%\begin{equation}\label{}
%p^{\frac{\alpha}{\alpha-1}Nr - 2r\beta^{*}_{1}-(Nr+\beta^{*}_1r+r-1))=} = p^J
%\end{equation}
%We have:
%\begin{equation}\label{}
%J=\frac{\alpha}{\alpha-1}\frac{N^{2}}{R} - 2\frac{N^2}{R^2}-\frac{N^2}{R}-\frac{N^2}{R^2}-\frac{N}{R}+1
%\end{equation}
%so 
%\begin{equation}\label{}
%J= (\frac{N^2}{R}) (\frac{1}{\alpha -1} -\frac{3}{R} -\frac{1}{N} + \frac{R}{N^2})
%\end{equation}
%choosing $K$ s.t. $\frac{3}{R} \le \frac{1}{\alpha -1}$ and $N$ large enough we can tell that $J$ behaves like $p^{tN^{2} + O(N)}$.

%For sufficient large $N$ this is a contradiction.
\vskip3mm
Part B.
We show that $(Z_{p^{\alpha}})^{\mathbb{N}} \not\prec_{loc} \oplus_{j=1}^{M} (Z_{(q_j)^{s_j}})^{\mathbb{N}}$. For the reader convinience we present the proof of $(Z_{2^{2}})^{\mathbb{N}} \not\prec_{loc} (Z_{2})^{\mathbb{N}} \oplus (Z_{3})^{\mathbb{N}}$. The general case differs only by more complicated notation.

Suppose to the contrary that $(Z_{4})^{\mathbb{N}} \prec^C (Z_{2})^{\mathbb{N}} \oplus (Z_{3})^{\mathbb{N}}$ for some $C>0$.
Let $G=(Z_{4})^{N}<(Z_{4})^{\mathbb{N}} $, and $\sigma_n:G \rightarrow  (Z_{2})^{\mathbb{N}} \oplus (Z_{3})^{\mathbb{N}}$ be injection from the definiton of local representation. Then by Lemma 1 we know that
$\sigma_{n} (G) = U$ s.t.:

$${\pmb{1}_{U}}= \sum_{i=1}^{l_2} \pmb{1}_{A_i} - \sum_{j=1}^{l_1}\pmb{1}_{ B_j}$$
where $l_1+l_2\le L$ and 
\begin{equation}\label{6}
\# U = \sum \# A_i - \sum \# B_j
\end{equation}
Hence 
\begin{equation}\label{7}
2^{k} = \sum_{i=1}^{l} x_i 
\end{equation}
where $l\le 2L$ and  $|x_i|=2^{k_i} 3^{q_i}$ (because any finite coset of $ (Z_{2})^{\mathbb{N}} \oplus (Z_{3})^{\mathbb{N}}$ has cardinality of this form). Notice that we can always assume, that cardinalities appering in $\eqref{6}$ are finite. Let $s' = \min_i \{ k_i \}$. Then dividing $\eqref{7}$ by $2^{s'}$ we have:
\begin{equation}\label{8}
2^{k-s'}  = \sum_{j=1}^{l} y_j
\end{equation}
where $|y_j|= 2^{r_j}3^{q_j}$ for $j=1,2,\dots, l$ and $gcd(y_1,y_2,...,y_l,2)=1$. By Theorem 4, the equation $\eqref{8}$ has only finite number of solutions with non vanishing proper subsums. Hence the sum $\sum_{j=1}^{l} y_j$ can take only finite number of values. Let $M=\max_{j=1,\dots,l} 3^{q_j}$. Then any coset A appering in $\eqref{6}$ has caridnality of the form $2^k M'$ for some $M' \le M$ and $M'$ is not divisible by 2. Then if follows that 
$$
A=\bigcup_{i=1}^{M'} R_j
$$
where $R_j$ are cosets with $\# R_j = 2^{k_j}$. Therefore we get that $\sigma_{n}(G )= U$ where 
$$
\pmb{1}_{U}= \sum_{i=1}^{r_1} \pmb{1}_{A_i} - \sum_{j=1}^{r_2}\pmb{1}_{B_j}
$$
where $r_1 + r_2 \le ML$ and $\# A_i$, $\# B_j$ are powers of 2. By Lemma 1 there exist coset $H \subset U$ with $\# H > 2^{K-\Lambda}$. By similar considerations we derive that there exist $G' \subset coset_s (G)$, where $s>2^{K-2\Lambda}$ and 
$\sigma_{n} (H) \supset G'$. By Lemma 2 there exist coset $G'' \subset G'$ s.t.
$G'' $ is isomorphic to $(Z_4)^{\frac{k}{2} - \lambda}$. Since $\sigma_{n}(G'') \subset H \approx (Z_2)^N$ we are in the position to apply part A.

\end{proof}

{\bf Remark} In the Proof of Part B we do not use Theorem 4 in its full strength. In fact we
use only the earlier result of  van der Poorten and Schlickewei (cf.[vdPS] ) and Evertse (cf. [E1]) asserting that the equation from Theorem 4 has only  finitly many solutions. Then our bound depends on the maximum over all those solutions.
Any explicit estimate of this maximum translates immediatly onto the estimation of the growth of the desired norm.
However no explicit estimate on this maximum is known.

\section{Appendix}

In the appendix we will estimate the number of cosets of given rank in groups 
 $Z^{N}_{p^{\alpha}}$. The result which we will use, can be found in [D].

\begin{defi}
A partition $\pmb{\alpha} =(\alpha_1, \alpha_2,...)$ is sequence of non-negative integers in decresing order, which contains only finitely many non-zero terms.

\end{defi}
Let $\| \pmb{\alpha} \|  = \sum_{i \ge 1} \alpha_i$ We say that finite abelian group A of rank $p^{\| \pmb{\alpha} \|}$ is of type $\pmb{\alpha} =(\alpha_1, \alpha_2,...)$  if A is isomorphic to the direct sum 
$$
Z_{p^{\alpha_1}} \oplus Z_{p^{\alpha_2}}\oplus ... \oplus Z_{p^{\alpha_k}}
$$
where the number of all non-zero terms in $\alpha$ equals to $k$ and is denoted by $l(\alpha)$.
\\
Let $\Theta$ be the set of all posible partitions. For $\pmb{\alpha}, \pmb{\beta} \in \Theta$ we introdue a relation  $\pmb{\beta} \subseteq \pmb{\alpha}$ if for every natural number $i$ holds $\beta_i \le \alpha_i$.

\begin{defi}
To every partition $\pmb{\beta}$ we can assign $\pmb{\beta}^{*}=(\beta^{*}_1, \beta^{*}_2,...)$ where $\beta^{*}_i$ is the number of $\beta_j$-s, for which the inequality $i \le \beta_j$ holds. Partition $\pmb{\beta}^*$ is called conjugate to $\pmb{\beta}$. Every partition determinates uniquely it`s conjugate partition.
\end{defi}
Notice that for every partition $\beta$,
$$
\| \pmb{\beta} \|   = \| \pmb{\beta}^\ast \|  
$$
Let the abelian group A be the group of type $\pmb{\alpha}$.
%Denote conjugate partition to $\pmb{\alpha^*}$ by $\pmb{\alpha^*}=(\alpha^*_1, \alpha^*_2,...)$.

\begin{defi}
Let $\pmb{B}_A (r):=\{ \pmb{\beta^*} \in \Theta \mid \pmb{\beta} \subseteq \pmb{\alpha}, \| \pmb{\beta}\|=r\}$,   $r\in\mathbb{N}$. 
\end{defi}

\begin{defi}
Let $\pmb{N}_A (r)$ be the number of subgroups of A, which are of rank $p^r$.
\end{defi}
By [D] we get:
\begin{lemat}(On number of subgroups)
$$
\pmb{N}_A (r)= \sum_{\pmb{\beta^*} \in \pmb{B}_A (r)} 
\prod_{i=1}^{\alpha_1} {\alpha^{*}_i - \beta^{*}_{i+1} \choose \beta^{*}_{i}-\beta^{*}_{i+1}}_p p^{(\alpha^{*}_i - \beta^{*}_i)\beta^{*}_{i+1}}
$$
where ${n \choose m}_p = \prod_{i=1}^{m}\frac{p^{n-m+i} - 1}{p^i -1}$ for $1\le m \le n$ and 0 for other values.
\end{lemat}
Within the appendix we will assume that $r<N$.
In the case considered in the paper it is enough to establish the upper estimation on the number of subgroups of given rank in the groups of type $\pmb{\alpha}=(\alpha,\alpha,...,\alpha,0,0...)$. We will denote such type by $(\alpha,\alpha,...,\alpha)$ where the number of $\alpha$ equals to $N$. Notice that in this situation we have $\pmb{\alpha^*}=(N,N,...,N)$ where  $l(\pmb{\alpha^*})=\alpha$.

In what follows we will estimate the upper bound on the number of subgroups of rank r in abelian group A of type $\pmb{\alpha}=(\alpha,\alpha,..,\alpha)$.
Notice that
$$
{\alpha^{*}_i - \beta^{*}_{i+1}\choose \beta^{*}_i - \beta^{*}_{i+1}}_p=\prod_{j=1}^{\beta^{*}_i-\beta^{*}_{i+1}} \frac{p^{j +\alpha^{*}_i-\beta^{*}_{i+1}+\beta^{*}_{i+1}-\beta^{*}_{i}} -1}{p^j -1}\le
\prod_{j=1}^{\beta^{*}_i-\beta^{*}_{i+1}} \frac{p^{j +\alpha^{*}_i-\beta^{*}_{i}}}{p^j -1} 
$$
For every prime number $p$ and natural number $j$ we have $p^j -1 \ge p^{j-1}$, so
$$
\prod_{j=1}^{\beta^{*}_i-\beta^{*}_{i+1}} \frac{p^{j +\alpha^{*}_i-\beta^{*}_{i}}}{p^j -1} \le \prod_{j=1}^{\beta^{*}_i-\beta^{*}_{i+1}} \frac{p^{j +\alpha^{*}_i-\beta^{*}_{i}}}{p^{j-1} } 
$$
and
$$
\prod_{i=1}^{\alpha} {\alpha^{*}_i - \beta^{*}_{i+1}\choose \beta^{*}_i - \beta^{*}_{i+1}}_p  \le  p^{B+C-D-E}
$$
where:
$$
B=\sum_{i=1}^{\alpha}(\sum_{j=1}^{\beta^{*}_i - \beta^{*}_{i+1} }j  )
$$

$$
C=\sum_{i=1}^{\alpha}(\sum_{j=1}^{\beta^{*}_i - \beta^{*}_{i+1} } \alpha^{*}_i  )
$$

$$
D=\sum_{i=1}^{\alpha}(\sum_{j=1}^{\beta^{*}_i - \beta^{*}_{i+1} } \beta^{*}_i)
$$

$$
E=\sum_{i=1}^{\alpha}(\sum_{j=1}^{\beta^{*}_i - \beta^{*}_{i+1} }j-1)
$$

Now we estimate above sums.
\begin{equation}\label{9}
\aligned
B=&\sum_{i=1}^{\alpha}(\sum_{j=1}^{\beta^{*}_i - \beta^{*}_{i+1} }j  )\le 
\sum_{i=1}^{\alpha}(\sum_{j=1}^{\beta^{*}_i - \beta^{*}_{i+1} }r  )\\ =&r(\beta^{*}_1-\beta^{*}_2+\beta^{*}_2-\beta^{*}_3-\dots
-\beta^{*}_{\alpha}+\beta^{*}_{\alpha}-\beta^{*}_{\alpha+1})=\beta^{*}_1 r
\endaligned
\end{equation}
because $\beta^{*}_i - \beta^{*}_{i+1}\le r$, $\| \pmb{\beta^{*}} \|  = r$, and non-zero $\beta^{*}_i$ can appear only on at most $\alpha$ places.

\begin{equation}\label{10}
C=\sum_{i=1}^{\alpha}(\sum_{j=1}^{\beta^{*}_i - \beta^{*}_{i+1} } \alpha^{*}_i  )\le N\beta^{*}_1
\end{equation}

This is a consequence of a fact that for $1\le i \le \alpha$ holds $a_i=N$, for other indeces $\alpha_i= 0$. We estimate $D=\sum_{i=1}^{\alpha}(\sum_{j=1}^{\beta^{*}_i - \beta^{*}_{i+1} } \beta^{*}_i)$ and $E=\sum_{i=1}^{\alpha}(\sum_{j=1}^{\beta^{*}_i - \beta^{*}_{i+1} }j-1)$ by 0.
For the exponential part of summand we have:
\begin{equation}\label{11}
\prod_{i=1}^{\alpha}  p^{(a_i - b_i)b_{i+1}}=p^{F}
\end{equation}
where
\begin{equation}\label{12}
F= \sum_{i=1}^{\alpha} \alpha^{*}_i \beta^{*}_{i+1} -  
\sum_{i=1}^{\alpha} \beta^{*}_i \beta^{*}_{i+1} \le \sum_{i=1}^{\alpha} \alpha^{*}_i \beta^{*}_{i+1} = N(r-\beta^{*}_1).
\end{equation}
This follows from the fact, that $\beta^{*}_i$ are non-negative, $\alpha^{*}_i=N$ for $1\le i \le \alpha $ and $r-\beta^{*}_1 = \sum_{j=1}^{\alpha} \beta^{*}_{j+1}$.
By $\eqref{9}$ , $\eqref{10}$, $\eqref{11}$ and $\eqref{12}$, every summand in Lemma 5 can be estimated by $p^{B+C+F}\le p^{Nr+\beta^{*}_1r}$.
The number of terms in the sum equals to number of partitions of $r$ on sums. This is less then the number of all posible representation of $r$ as ordered sum, which in turn is equal to
$2^{r-1}$. Therefore:
$$
\pmb{N}_A (r) \le 2^{r-1}p^{Nr+\beta^{*}_1 r}\le p^{Nr+\beta^{*}_1r+r-1}
$$
Let now $G$ be a group of rank $p^{N \alpha}$ and type $\pmb{\gamma}=(\gamma_1, \gamma_2,...)$ where $\gamma_i \le \alpha$.  We are going to minorize the number of subgroups of $G$ of rank $p^r$. In order to do this, we are looking for the minimal value of $l(\gamma)$, under the above assumption. Obviously:
$$
l(\pmb{\gamma)} \cdot\max\{ \gamma_i \} \ge \|  \pmb{\gamma} \|
$$
from the assumptions we get $\max\{ \gamma_i \} \le \alpha -1 $ and $\| \pmb{\gamma}\|=N\alpha$. Hence 
$$
l(\pmb{\gamma)} \ge \frac{\alpha}{\alpha -1} N
$$
It appears that to get the required estimation from below it is sufficient to consider only one
summand from Lemma 5, i.e.

$$
\sum_{\pmb{\beta^*} \in \pmb{B}_A (r)} 
\prod_{i=1}^{\alpha_1} {\alpha^{*}_i - \beta^{*}_{i+1} \choose \beta^{*}_{i}-\beta^{*}_{i+1}}_p p^{(\alpha^{*}_i - \beta^{*}_i)\beta^{*}_{i+1}} \ge 
\prod_{i=1}^{\alpha_1} {\alpha^{*}_i - \widetilde{\beta}^{*}_{i+1} \choose \widetilde{\beta}^{*}_{i}-\widetilde{\beta}^{*}_{i+1}}_p p^{(\alpha^{*}_i - \widetilde{\beta}^{*}_i)\widetilde{\beta}^{*}_{i+1}} 
$$
%where $\widehat{\pmb{\beta^*}}$ is such fixed partition that $\widehat{\beta^{*}_1} = r$.
%That partition can be chosen becous $r< N$. In this situation, we have
where $\widetilde{\pmb{\beta}}^*=(r,0,0,...)$ and $\widetilde{\pmb{\beta}}=(1,1,1,...,1)$ and the number of non zero terms in the second partition equals  $r$. 
Similarly as in previous estimation we have:
$$
\prod_{i=1}^{\alpha_1} {\alpha^{*}_i - \widetilde{\beta}^{*}_{i+1} \choose \widetilde{\beta}^{*}_{i}-\widetilde{\beta}^{*}_{i+1}}_p p^{(\alpha^{*}_i - \widetilde{\beta}^{*}_i)\widetilde{\beta}^{*}_{i+1}}\ge p^{E+C-D-B}
$$
where
$$
B=\sum_{i=1}^{\alpha}(\sum_{j=1}^{\beta^{*}_i - \beta^{*}_{i+1} }j  )
$$

$$
C=\sum_{i=1}^{\gamma^{*}_1}( \sum_{j=1}^{\widetilde{\beta}^{*}_{i}}-\widetilde{\beta}^{*}_{i+1} \gamma^{*}_{i})
$$

$$
D=\sum_{i=1}^{\alpha} (\sum_{j=1}^{\beta^{*}_{i}-\beta^{*}_{i+1}} \beta^{*}_{i})
$$

$$
E=\sum_{i=1}^{\alpha}(\sum_{j=1}^{\beta^{*}_i - \beta^{*}_{i+1} }j-1)
$$
We estimate $E$ by 0. For the rest we have:
$$
D=\sum_{i=1}^{\alpha} (\sum_{j=1}^{\beta^{*}_{i}-\beta^{*}_{i+1}} \beta^{*}_{i})\le \beta^{*}_{1} r
$$
which follows from $\beta^{*}_{i}\le r$ and the fact that $\sum_{i=1}^{\alpha} (\sum_{j=1}^{b_i - b_{i+1}} 1 )= b_1$. Further,
$$
\aligned
B=&%\sum_{i=1}^{\alpha}(\sum_{j=1}^{\beta^{*}_i - \beta^{*}_{i+1} }j  )\le 
\sum_{i=1}^{\alpha}(\sum_{j=1}^{\beta^{*}_i - \beta^{*}_{i+1} }r  )\\
=&r(\beta^{*}_1-\beta^{*}_2+\beta^{*}_2-\beta^{*}_3-\dots
-\beta^{*}_{\alpha}+\beta^{*}_{\alpha}-\beta^{*}_{\alpha+1})=\beta^{*}_1 r
\endaligned
 $$
%Because $E$ is non negative we bound $E$ by 0.
 For  $C$ we get
$$
C=\sum_{i=1}^{\gamma^{*}_1}( \sum_{j=1}^{\widetilde{\beta}^{*}_{i}}-\widetilde{\beta}^{*}_{i+1} \gamma^{*}_{i})\ge \widetilde{\beta}^{*}_{1} \frac{\alpha}{\alpha -1}N
$$
For our choice of $\beta$,:
$$
\prod_{i=1}^{\alpha}  p^{(\alpha^{*}_i - \widetilde{\beta}^{*}_i)\widetilde{\beta}^{*}_{i+1}}=
1
$$
Taking into account the above estimation we get

$$
\pmb{N}_G (r)\ge p^{\frac{\alpha}{\alpha-1}Nr - 2r\beta^{*}_{1}}
$$
In our case we take $r=\beta^{*}_{1}=\frac{N}{K}$ where  $K$ is chosen suitable to our purpose. For groups $A_1, A_2$ of rank $p^{\alpha N}$ where group $A_1$ is of type $(\alpha, \alpha,...,\alpha)$ and $A_2$ is of type $(\gamma_1, \gamma_2,...)$ where $\gamma_i \le \alpha-1$ for every $i$, we have:

$$
\pmb{N}_{A_1}(r) \le p^{Nr+\beta^{*}_1r+r-1} 
$$
$$
\pmb{N}_{A_2}(r)\ge  p^{\frac{\alpha}{\alpha-1}Nr - 2r\beta^{*}_{1}}
$$

Having estimated the number of subgroups we can now estime the number of related cosets.

By the Lagrange theorem, the number of cosets in $G$ generated by given subgorup $H<G$ equals $\frac{\# G}{\# H}$. Obviously different group generates different cosets. So the number of cosets of given rank in $G$ equals to the number of subgroups of this rank in $G$ multiplied by $\frac{\# G}{\# H}$.

The above considerations applied to $A_1, A_2$ gives that the number of cosets of rank $p^r$ in $A_2$, $A_1$  has respectively lower and upper estimations $p^{\frac{\alpha}{\alpha-1}Nr - 2r\beta^{*}_{1}+O(N)}$, $p^{Nr+\beta^{*}_1r+r-1 +O(N)}$
and the lemma follows.


\begin{thebibliography}{HD}
\normalsize
\baselineskip=17pt
\bibitem[ARS]{ars} M. Auslander, I. Reiten, S. O. Smalo \textit{Representation Theory of Artin Algebras} Cambridge Univ. Press, 1997 
\bibitem[D]{dju} Djubjuk P.E., \textit{On the number of subgroups of a finite abelian group} Izv. Akad. Nauk SSSR Ser. Mat. 12, (1948) 351 - 378.
\bibitem [GS]{gs}  B: Green, T. Sanders, \textit{A quantitative version of the idempotent theorem in Harmonic Analysis}, Ann. Math. vol 168 (2008), 1025-1054
\bibitem[E1]{e1} J.-H. Evertse, \textit{On sums of S-units and linear recurrences}, Compos. Math. vol 53 (1984), 225-244
\bibitem[E2]{e2}J.-H. Evertse, \textit{The number of solutions of decomposable form equations}, Invent. Math. vol 122 (1995), 559–601.
%\bibitem [ESS]{ess} J.-H Evertse, H. P. Schilckewei, W. M. Schmidt: \textit{Linear equations in variables which lie
%in a multiplicative group}, Ann. Math. vol 155 (2002), 807 - 836 
\bibitem[P]{p} A. Pelczynski \textit{Selected problems on the structre of complemented subspaces of Banach spaces}, in: Methods in Banach Spaces, edited by Jesus M. F. Castillo and William B. Johnson, London Mathematical Society Lecture Note Series 337, Cambridge University Press (2006), 341–354
\bibitem [vdPS]{vdps} A. J. van der Poorten, A. J. Sclickewei \textit{The growth condition for recurrence seguences} Macquarie Univ. Math. Rep. 82-0041, North Ryde, Australia, 1982
\bibitem[Ro]{ro} J. J. Rotman, \textit{An Introduction to the Group Theory}, Springer, 1995
\bibitem [Ru]{ru} W. Rudin \textit{Fourier Analysis on Groups}, Interscience, Wiley (1962).
\bibitem[V]{v} N. Vilenkin, \textit{On a class of complete orthonormal systems}, Bull. Acad. Sci. URSS. Sér. Math. [Izvestia Akad. Nauk SSSR] 11 (1947), 363–400 (Russian, with English summary)
\bibitem[W]{w} M. Wojciechowski \textit{The non-equivalence between the trigonometric system and the system of functions with pointwise restrictions on values in the uniform and $L^1$ norms}, Math. Proc. Camb. Phil. Soc (2011)
\end{thebibliography}
\end{document}